\documentclass[a4paper,10pt,onecolumn,twoside]{article}
\usepackage{mathrsfs}
\usepackage{mathrsfs}
\usepackage{amsfonts}
\usepackage{fancyhdr}
\usepackage{amsmath,amsfonts,amssymb,graphicx,amsthm}
\allowdisplaybreaks

\usepackage{amsmath,latexsym,amsfonts,indentfirst,amsthm,amsxtra,amssymb,bm}

\usepackage{amssymb,mathrsfs,graphicx,enumerate}
\numberwithin{equation}{section}

 \newtheorem{lemma}{Lemma}[section]
 \newtheorem{proposition}[lemma]{Proposition}
 \newtheorem{theorem}{Theorem}[section]

 \theoremstyle{remark}
 \newtheorem{remark}{Remark}[section]
\numberwithin{equation}{section}

\begin{document}

%--------------------------------------------------------------------------------------
%--------------------------------------------------------------------------------------
\title{\bf Radially Symmetric Stationary Wave for the Exterior Problem of Two-dimensional Burgers Equation}
\author{{\bf Huijiang Zhao}\footnote{Email address: hhjjzhao@whu.edu.cn}\quad and\quad {\bf Qingsong Zhao}\footnote{Email address: qszhao@whu.edu.cn}\\[1mm]
School of Mathematics and Statistics, Wuhan University, China\\
Hubei Key Laboratory of Computational Science, Wuhan University, China}

\date{}

\maketitle

% \thispagestyle{style}                             % 当前页的页面式样, style=empty 当前页不显示页眉页脚, style=first

%\begin{center}
%{\bf Dedicated to Professor Shuxing Chen on the occasion of his 80th birthday}
%\end{center}

\begin{abstract}
We are concerned with the radially symmetric stationary wave for the exterior problem of two-dimensional Burgers equation. A sufficient and necessary condition to guarantee the existence of such a stationary wave is given and it is also shown that such a stationary wave satisfies nice decay estimates and is time-asymptotically nonlinear stable under radially symmetric perturbation.

\end{abstract}

\maketitle

%\tableofcontents

%---------------------------------------------------------------------------------------
\section{Introduction and main results}

In this paper, we consider the problem on the precise description of the large time behaviors of global smooth solutions to the following initial-boundary value problem of multidimensional Burgers equation in an exterior domain $\Omega:=\mathbb{R}^n\backslash \overline{\mathbb{B}}_{r_0}(0)\subset\mathbb{R}^n$ for $n\geq 2$:
\begin{eqnarray}\label{1.1}
   {\bf u}_t+({\bf u}\cdot\nabla){\bf u}&=&\mu\Delta {\bf u}, \quad t>0,\ x\in \Omega,\nonumber\\
   {\bf u}(0,x)&=&{\bf u}_0(x),\quad x\in\Omega,\\
   {\bf u}(t,x)&=&{\bf b}(t,x),\quad t>0,\ x\in\partial\mathbb{B}_{r_0}(0),\nonumber
\end{eqnarray}
and, as in \cite{Hashimoto-NonliAnal-2014, Hashimoto-OsakaJMath-2016, Hashimoto-Matsumura-JDE-2019}, our main purpose is to understand how the space dimension $n$ effect the large time behaviors of solutions of the initial-boundary value problem \eqref{1.1}. Here ${\bf u}=\left(u_1(t,x),\cdots,u_n(t,x)\right)$ is a vector-valued unknown function of $t\in\mathbb{R}^+$ and $x=(x_1,\cdots,x_n)\in \mathbb{R}^n$, ${\bf u}\cdot\nabla=\sum\limits_{j=1}^nu_j\frac{\partial}{\partial x_j}$, $\mu>0$ and $r_0>0$ are some given constants. ${\bf u}_0(x)$ and ${\bf b}(t,x)$ are given initial and boundary values respectively satisfying the compatibility condition ${\bf b}(0,x)={\bf u}_0(x)$ for all $x\in\partial\mathbb{B}_{r_0}(0)$.

Throughout this paper, we will concentrate on the radially symmetric solutions for the initial-boundary value problem \eqref{1.1}. For such a case, under the assumption that ${\bf b}(t,x)=\frac{x}{|x|}v_-, {\bf u}_0(x)=\frac{x}{|x|}v_0(|x|)$ satisfying $\lim\limits_{|x|\to+\infty}v_0(|x|)=v_+$ and $v_0(r_0)=v_-$ for some given constants $v_\pm\in\mathbb{R}$ with $v_0(|x|)$ being some given scalar function, then if we introduce a new unknown function $v(t,r)$ by letting ${\bf u}(t,x)=\frac xrv(t,r)$ with $r=|x|$, we can deduce that $v(t,r):= v(t,|x|)$ solves the following initial-boundary value problem
\begin{eqnarray}\label{1.2}
    v_t+\left(\frac{v^2}{2}\right)_r&=&\mu\left(v_{rr}+\left(\frac {(n-1)v}{r}\right)_r\right),\quad t>0,\ r>r_0,\nonumber\\
    v(t,r_0)&=&v_-,\quad \lim\limits_{r\rightarrow \infty }v(t,r)=v_+,\quad t>0,\\
    v(0,r)&=&v_0(r),\quad r>r_0,\nonumber
\end{eqnarray}
where the initial data $v_0(r)$ is assumed to satisfy the compatibility condition
\begin{equation}\label{Compatibility condition}
v_0(r_0)=v_-,\quad \lim\limits_{r\to\infty}v_0(r)=v_+.
\end{equation}
Throughout the rest of this paper, we set $V_-=v_--\frac{\mu(n-1)}{r_0}$.

It is well-known that, cf. \cite{Liu-Matsumura-Nishihara-SIMA-1998, Liu-Nishihara-JDE-1997, Liu-Yu-ARMA-1997, Matsumura-MAA-2001, Nishihara-JMAA-2001, Nishihara-ADC-2001} and the references cited therein, to give a precise description of the large time behaviors of global solutions $v(t,r)$ to the initial-boundary value problem \eqref{1.2}, in addition to the rarefaction waves and viscous shock waves, which are sufficient to describe the asymptotics of the global solutions for the corresponding Cauchy problem of one-dimensional scalar viscous conservation laws, a new type of nonlinear wave, i.e. the so-called {\it stationary wave} $\phi(r)$ solving the following problem
\begin{eqnarray}\label{1.4}
\frac{d}{dr}\left(\frac{\phi^2}{2}\right)&=&\mu\left(\phi_{rr}+\left(\frac{(n-1)\phi}{r}\right)_r\right),\quad r>r_0,\\
\phi(r_0)&=&v_-,\ \lim\limits_{r\to\infty}\phi(r)=v_+,\nonumber
\end{eqnarray}
should be introduced, which is due to the appearance of the boundary condition \eqref{1.2}$_2$.

The main purpose of this paper focuses on the existence and time-asymptotically nonlinear stability of such a stationary wave $\phi(r)$ under radially symmetric perturbation. To this end, if one integrates \eqref{1.4}$_1$ with respect to $r$ from $r$ to $\infty,$ then the problem \eqref{1.4} is rewritten as
\begin{eqnarray}\label{1.5}
    \frac{d\phi}{dr}+\frac{n-1}{r}\phi&=&\frac1{2\mu}\left(\phi^2-v_+^2\right), \quad r>r_0,\\
    \phi(r_0)&=&v_-, \ \lim\limits_{r\to\infty}\phi(r)=v_+.\nonumber
\end{eqnarray}
Moreover, if we set
\begin{equation}\label{1.6}
\psi(r)=\phi(r)-\frac{\mu(n-1)}{r},
\end{equation}
then we can get from \eqref{1.5} that $\psi(r)$ solves
\begin{eqnarray}\label{1.7}
\frac{d\psi}{dr}&=&\frac1{2\mu}\left(\psi^2-v_+^2\right)-\frac{\mu(n-1)(n-3)}{2r^2}, \quad r>r_0,\nonumber\\
\psi(r_0)&=&V_-:=v_--\frac{\mu(n-1)}{r_0}, \ \lim\limits_{r\to+\infty}\psi(r)=v_+.
\end{eqnarray}

Such a problem has been studied by I. Hashimoto and A. Matsumura in \cite{Hashimoto-NonliAnal-2014, Hashimoto-OsakaJMath-2016, Hashimoto-Matsumura-JDE-2019}, what they found for the multidimensional Burgers equation is that, unlike the one-dimensional case, the stationary wave $\phi(r)$ (or $\psi(r)$) satisfying \eqref{1.5} (or \eqref{1.7}) is, generally speaking, no longer monotonic. The results obtained in \cite{Hashimoto-NonliAnal-2014, Hashimoto-OsakaJMath-2016, Hashimoto-Matsumura-JDE-2019} can be summarized as in the following:
\begin{itemize}
\item When $v_+=0$, \eqref{1.5}$_1$ is Bernoulli type ordinary differential equation which can be solved explicitly, thus one can deduce that for $n=2$, \eqref{1.5} admits the following unique nontrivial solution $\phi_1^2(r)$
    \begin{equation}\label{1.8}
    \phi_1^2(r)=\frac{v_-}{r\left[\frac{1}{r_0}-\frac{v_-}{2\mu}\ln\frac{r}{r_0}\right]}
    \end{equation}
    for $r\geq r_0$ if and only if $v_-<0$, while for $n\geq 3$, \eqref{1.5} possesses the following unique nontrivial solution
    \begin{equation}\label{1.9}
    \phi_1^n(r)=\frac{v_-}{\left(1-\frac{r_0v_-}{2\mu(n-2)}\right)\left(\frac{r}{r_0}\right)^{n-1}
    +\frac{r_0v_-}{2\mu(n-2)}\frac{r}{r_0}}
    \end{equation}
    for $r\geq r_0$ if and only if $v_-\leq \frac{2\mu(n-2)}{r_0}$. These stationary waves $\phi_1^2(r)$ and $\phi_1^n(r) (n\geq 3)$ are monotonic and satisfy
    \begin{equation}\label{1.10}
    \left|\phi_1^n(r)\right|\lesssim \left\{
    \begin{array}{rl}
    \frac{1}{r|\ln r|},\quad& n=2,\\[2mm]
    r^{1-n},\quad& n\geq 3
    \end{array}
    \right.
    \end{equation}
    for $r\geq r_0$ and are time-asymptotic nonlinear stability, cf. \cite{Hashimoto-NonliAnal-2014, Hashimoto-OsakaJMath-2016, Hashimoto-Matsumura-JDE-2019}. Here and in the rest of this paper $f(r)\lesssim g(r)$ means that there exists a generic positive constant $C$ such that $|f(r)|\leq C|g(r)|$ holds for $r\geq r_0$. $f(r)\sim g(r)$ if $f(r)\lesssim g(r)$ and $g(r)\lesssim f(r)$;

\item When $n=3$, \eqref{1.7}$_1$ is an autonomous ordinary differential equation and,
similar to the one-dimensional case, it can be solved explicitly also. In fact, it is  easy to see that \eqref{1.7} admits a unique nontrivial solution
\begin{equation}\label{1.11}
\psi^S(r)=\frac{v_+\left(1-\frac{|v_+|+V_-}{|v_+|-V_-}\exp\left(-\frac{|v_+|(r-r_0)}{\mu}\right)\right)} {1+\frac{|v_+|+V_-}{|v_+|-V_-}\exp\left(-\frac{|v_+|(r-r_0)}{\mu}\right)}
\end{equation}
for $r\geq r_0$ if and only if $v_+<0, V_-<|v_+|$. Such a stationary wave $\psi^S(r)$ satisfies
$$
\left|\psi^S(r)-v_+\right|\lesssim \exp\left(-\frac{|v_+|r}{\mu}\right),\quad r\geq r_0
$$
and is also shown to be nonlinear stable under radially symmetric perturbation, cf. \cite{Hashimoto-Matsumura-JDE-2019};

\item For the case when one can not deduce an explicit formula for the solutions of \eqref{1.5} or \eqref{1.7}, the result available up to now focuses on the case $n\geq 4$. In such a case, it is shown in \cite{Hashimoto-Matsumura-JDE-2019} that if
\begin{equation}\label{1.12}
v_+<0, V_-\leq|v_+|,
\end{equation}
then one can deduce that $\psi^S(r)$ is a upper bound of the solution $\psi(r)$ of \eqref{1.7}, while
\begin{eqnarray}\label{1.13}
\psi_S^n(r)&=&v_++\left(V_--v_+\right)\exp\left(-\frac{|v_+|(r-r_0)}{\mu}\right)\\
&&-\frac{\mu(n-1)(n-3)}{2}\int_{r_0}^r\frac1{s^2}\exp\left(-\frac{|v_+|(r-s)
     }{\mu}\right)ds,\quad r\geq r_0\nonumber
\end{eqnarray}
gives the lower bound, from which one yield the existence of stationary wave $\psi(r)$ to \eqref{1.7}, which satisfies
\begin{equation}\label{1.14}
\left|\psi(r)-v_+\right|\lesssim r^{-2}
\end{equation}
for $r\geq r_0$. Although it is no longer monotonic, its nonlinear stability is justifies in \cite{Hashimoto-Matsumura-JDE-2019} for $v_\pm<0, V_-<v_+$ and is later extended in \cite{Yang-Zhao-Zhao-2019} to cover the case when \eqref{1.12} holds.
\end{itemize}

Even so, for the two-dimensional case, to the best of our knowledge, the only result available up to now is on the case when $v_+=0$ and in such a case, the unique solution $\phi_1^2(r)$ to \eqref{1.5} is given by \eqref{1.8}. Thus it is an interesting problem to see what happens when $v_+<0$ and the main purpose of this paper is concentrated on such a problem.

Throughout the rest of this paper, we will focus on the case $n=2$ and in such a case, \eqref{1.7} can be rewritten as
\begin{eqnarray}\label{1.15}
\frac{d\psi}{dr}&=&\frac1{2\mu}\left(\psi^2-v_+^2+\frac{\mu^2}{r^2}\right), \quad r>r_0,\nonumber\\
\psi(r_0)&=&V_-:=v_-+\frac{\mu}{r_0}, \ \lim\limits_{r\to+\infty}\psi(r)=v_+.
\end{eqnarray}

For the solvability of the problem \eqref{1.15}, we have the following result

\begin{theorem} \label{Thm1.1} There exists a constant $a_*$ satisfying
\begin{equation}\label{1.16}
a_*\in \left\{
\begin{array}{ll}
\left(\sqrt{v_+^2-\frac{\mu^2}{r_0^2}},|v_+|\right),\quad &{\textrm{if}}\ \ v_+\leq -\frac{\mu}{r_0},\\[2mm]
\left[-\frac\mu{r_0},|v_+|\right),\quad & {\textrm{if}}\ \ -\frac\mu{r_0}<v_+<0
\end{array}
\right.
\end{equation}
such that \eqref{1.15} admits a unique solution $\psi(r)\in C^\infty([r_0,+\infty))$ if and only if
\begin{equation}\label{1.17}
v_+<0,\ \ V_-<a_*.
\end{equation}
Moreover such a solution $\psi(r)$ satisfies
\begin{equation}\label{1.18}
\left|\psi(r)-v_+\right|=\left|\phi(r)-v_+-\frac\mu r\right|\leq O(1)r^{-2},\quad r\geq r_0.
\end{equation}
Here $\phi(r)$ is the corresponding solution of the following problem
\begin{eqnarray}\label{1.19}
    \frac{d\phi}{dr}+\frac{\phi}{r}&=&\frac1{2\mu}\left(\phi^2-v_+^2\right), \quad r>r_0,\\
    \phi(r_0)&=&v_-, \ \lim\limits_{r\to\infty}\phi(r)=v_+.\nonumber
\end{eqnarray}
\end{theorem}
\begin{remark} From the proof of Theorem \ref{Thm1.1}, it is easy to see that, generally speaking, the unique solution $\psi(r)$ to the problem \eqref{1.15} is no longer monotonic. In fact, we can deduce from the proof of Theorem \ref{Thm1.1} that if $V_-=\sqrt{v_+^2-\frac{\mu^2}{r^2_0}}$, $\psi(r)$ is strictly monotonic decreasing on $r\geq r_0$, while if $\sqrt{v_+^2-\frac{\mu^2}{r^2_0}}<V_-<a_*$, $\psi(r)$ is firstly strictly monotonic increasing to its maximum, then is strictly monotonic decreasing and tends to $v_+$ as $r\to+\infty$.
\end{remark}

For the time-asymptotically nonlinear stability of the stationary wave $\phi(r)$ constructed in Theorem \ref{Thm1.1}, unlike the one-dimensional case, the main trouble is caused by the fact that such a stationary wave $\phi(r)$ is no longer monotonic and to overcome such a difficulty, as in \cite{Yang-Zhao-Zhao-2019}, we use the anti-derivative method by introducing the new unknown function
\begin{equation}\label{1.20}
  w(t,r)=-\int_r^\infty(v(t,y)-\phi(y))dy,
\end{equation}
then \eqref{1.20}, \eqref{1.15} together with \eqref{1.19} tell us that $w(t,r)$ solves
\begin{eqnarray}\label{1.21}
    w_t+\psi w_r-\mu w_{rr}&=&-\frac12 w_r^2,\quad r>r_0,\ t>0,\nonumber\\
    w_r(t,r_0)&=&\lim\limits_{r\to\infty}w_r(t,r)=\lim\limits_{r\to\infty}w(t,r) =0,\quad t>0,\\
    w(0,r)&=&w_0(r):=-\int_r^\infty\left(v_0(y)-\phi(y)\right)dy,\quad r>r_0.\nonumber
\end{eqnarray}

With the above preparations in hand, we now turn to state our result on the nonlinear stability of the stationary wave $\phi(r)$ constructed in Theorem \ref{Thm1.1}. In fact, motivated by \cite{Yang-Zhao-Zhao-2019}, if we introduce the weight function
$\chi:[r_0,\infty)\rightarrow\mathbb{R}$
\begin{equation}\label{1.22}
  \chi(r)=\exp\left(-\frac1\mu\int_{r_0}^r\psi(s)ds\right)\int_r^\infty
  \left(\frac2{r_0}-\frac1s\right)\exp\left(\frac1\mu\int_{r_0}^s \psi(\tau)d\tau\right)ds
\end{equation}
and by employing the weighted energy method as in \cite{Yang-Zhao-Zhao-2019} with a slight modification, we can get that
\begin{theorem}\label{Thm1.2}
  Suppose that the condition listed in Theorem \ref{Thm1.1} holds and $w_0\in H^2$ satisfying
  \begin{equation}\label{1.23}
    \|w_0\|^2_{L^2([r_0,+\infty))}\left(\|w_{0r}\|^2_{L^2([r_0,+\infty))} +\frac{C_Uv_+^2}{C_L\mu^2}\|w_0\|^2_{L^2([r_0,+\infty))}\right) \leq \frac{\mu^4C_L}{64C_U},
  \end{equation}
  then the initial-boundary value problem \eqref{1.21} admits a unique global solution $w(t,r)$ satisfying
  \begin{eqnarray*}
    w(t,r)&\in& C\left([0,\infty);H^2([r_0,+\infty))\right),\\
    \frac{\partial w(t,r)}{\partial r}&=&v(t,r)-\phi(r)\in  L^2\left([0,\infty); H^2([r_0,+\infty))\right)
  \end{eqnarray*}
  and
  \begin{equation}\label{1.24}
    \lim\limits_{t\to +\infty}\sup_{r\geq r_0}\big(|w(t,r)|+|v(t,r)-\phi(r)|\big)=0,
  \end{equation}
where $C_U:=\|\chi\|_{L^\infty([r_0,+\infty))}, C_L:=\inf\limits_{r\geq r_0}\chi(r)$ and from the estimate \eqref{1.18} and Lemma 2.2 of \cite{Yang-Zhao-Zhao-2019}, we know that $C_U$ and $C_L$ are some positive constants.
\end{theorem}
\begin{remark} Several remarks are listed below:
\begin{itemize}
\item [(i).] From the assumption \eqref{1.23} we imposed on the initial data $w_0(r)$, one can deduce that $\|w_0\|_{L^2([r_0,+\infty))}$ should be small, while $\|w_{0r}\|_{L^2([r_0,+\infty))}$ can be large. Even so, such a stability result is essentially a stability result with small initial perturbation. It would be interesting to see whether the radially symmetric stationary wave $\psi(r)$ constructed in Theorem \ref{Thm1.1} is nonlinear stable for large initial perturbation or not;
\item [(ii).] Theorem \ref{Thm1.2} shows that the radially symmetric stationary wave $\psi(r)$ constructed in Theorem \ref{Thm1.1} is time-asymptotically nonlinear stable under radially symmetric perturbation. An interesting problem is to see whether it is nonlinear stable or not under general multidimensional perturbation. For some recent progress on this problem for the case when $n\geq 3, v_+=0, v_-<\frac{2\mu(n-2)}{r_0}$, those interested is referred to \cite{Hashimoto-Math Nach-2020}.
\end{itemize}
\end{remark}

For the temporal convergence rates of the unique global solution $v(t,r)$ of the initial-boundary value problem (\ref{1.2}) toward the stationary wave $\phi(r)$, since $\phi(r)$ satisfies \eqref{1.18}, we have by repeating the argument used in \cite{Yang-Zhao-Zhao-2019} that
\begin{theorem}\label{Thm1.3}
  Under the assumptions stated in Theorem \ref{Thm1.2}, we can get that
  \begin{itemize}
  \item For any $\alpha>0,$ if we assume further that $w_0\in L^2_\alpha([r_0,+\infty)),$  then we have
  \begin{equation}\label{1.25}
    \sup_{r\geq r_0}|v(t,r)-\phi(r)|\lesssim \left(\|w_0\|_{H^2([r_0,+\infty))}+\left\|r^{\frac \alpha 2}w_0\right\|_{L^2([r_0,+\infty))}\right)(1+t)^{-\frac \alpha 2},\quad t\geq 0;
  \end{equation}
  \item For any $\beta$ and $\gamma$ satisfying
  \begin{equation}\label{1.26}
      0< \beta\leq \min\left\{\frac2{r_0},\frac{8}{\left(8C_U+1\right)r_0}\right\},\quad 0<\gamma\leq\frac{3\mu\beta}{8r_0C_U},
  \end{equation}
if we assume further that $w_0\in H^{2,\beta}_{\exp}([r_0,+\infty)),$ then there exists a time-independent positive constant $C>0$ such that
   \begin{equation}\label{1.27}
     \sup_{r\geq r_0}|v(t,r)-\phi(r)|\leq C\left\|e^{\beta r/2}w_0\right\|_{H^2([r_0,+\infty))}\exp(-\gamma t)
   \end{equation}
holds for all $t\geq 0$.
\end{itemize}
\end{theorem}

Now we outline the main ideas used to prove our main results. For the existence of radially symmetric stationary wave $\psi(r)$ to \eqref{1.15}, as in \cite{Hashimoto-Matsumura-JDE-2019}, one first considers the following Cauchy problem
\begin{eqnarray}\label{1.28}
\frac{d\psi}{dr}&=&\frac1{2\mu}\left(\psi^2-v_+^2-\frac{\mu^2(n-1)(n-3)}{r^2}\right), \quad r>r_0,\nonumber\\
\psi(r_0)&=&V_-:=v_-+\frac{\mu}{r_0}.
\end{eqnarray}
Let $\psi^n(r)$ be the unique local solution of \eqref{1.28} defined on the interval $[r_0,R)$, if one can find suitable lower bound $\widetilde{\psi}_S(r)$ and upper bound $\widetilde{\psi}^S(r)$ for $\psi(r)$ such that
\begin{itemize}
\item both $\widetilde{\psi}_S(r)$ and $\widetilde{\psi}^S(r)$ are well-defined for $r\geq r_0$;
\item $\left|\widetilde{\psi}_S(r)-v_+\right|+ \left|\widetilde{\psi}^S(r)-v_+\right|\lesssim r^{-2}$, $r\geq r_0$,
\end{itemize}
then one can easily deduce by the continuation argument that the Cauchy problem \eqref{1.28} admits a unique solution $\psi^n(r)$ which is defined on $r\geq r_0$ and belongs to $C^\infty([r_0,+\infty))$. Moreover such a $\psi^n(r)$ satisfies $\left|\psi^n(r)-v_+\right|\lesssim r^{-2}$ and thus it is indeed the desired solution of \eqref{1.7}.

The difference between the case $n\geq 4$ and the case $n=2$ lies in the way to construct the desired upper bound $\widetilde{\psi}^S(r)$ for the solution $\psi^n(r)$ of \eqref{1.28}. In fact, for each $n\geq 2$,  $\psi^n_S(r)$ defined by \eqref{1.13} always gives the desired lower bound for the solution $\psi^n(r)$ of \eqref{1.28}, while for $n=2$, $\psi^S(r)$ given by \eqref{1.11} is no longer a upper bound for the solution $\psi^n(r)$ of \eqref{1.28} with $n=2$. To overcome such a difficulty, for each $r_1\geq\max\left\{r_0, \frac\mu{|v_+|}\right\}$, we first consider the following auxiliary Cauchy problem
\begin{eqnarray}\label{1.29}
\frac{d\eta(r;r_1)}{dr}&=&\frac{1}{2\mu}\left(\eta^2(r;r_1)-v_+^2+\frac{\mu^2}{r^2}\right),\nonumber\\
\eta(r;r_1)\big|_{r=r_1}&=&\sqrt{v_+^2-\frac{\mu^2}{r_1^2}}
\end{eqnarray}
and we can show that
\begin{itemize}
\item [(i).] the Cauchy problem \eqref{1.29} admits a unique solution $\eta(r;r_1)$ on the interval $[r_0,+\infty)$;
\item [(ii).] $\eta(r;r_1)$ satisfies
\begin{equation}\label{1.30}
\left|\eta(r;r_1)-v_+\right|\lesssim r^{-2};
\end{equation}
\item [(iii).] $a(r_1)=\eta(r_0;r_1)$ is an increasing function of $r_1$ and is bounded from above by $|v_+|$, thus the limit $a_*=\lim\limits_{r_1\to+\infty}a(r_1)$ exists;
\item [(iv).] Based on the existence of such a constant $a_*$, we can then construct the desired upper bound $\widetilde{\psi}^S(r)$ for the solution $\psi^n(r)$ of \eqref{1.28} with $n=2$ and to show that $a_*$ is indeed the threshold value to guarantee the existence of the stationary wave $\psi(r)$ to \eqref{1.15}.
\end{itemize}
The basis of the above analysis is that
\begin{itemize}
\item suppose that the Cauchy problem \eqref{1.29} possesses a solution $\eta(r;r_1)$ defined on the interval $(R_1, R_2)$ with $r_0\leq R_1\leq r_1<R_2\leq+\infty$, then $\eta(r;r_1)$ is monotonic increasing for $R_1\leq r\leq r_1$ and monotonic decreasing for $r_1\leq r<R_2$.
\end{itemize}

For the time-asymptotically nonlinear stability of the stationary wave constructed in Theorem \ref{Thm1.1}, the main difficulty, as pointed out in \cite{Hashimoto-NonliAnal-2014, Hashimoto-OsakaJMath-2016, Hashimoto-Matsumura-JDE-2019}, is caused by the fact that such a stationary wave is no longer monotonic. Motivated by \cite{Yang-Zhao-Zhao-2019}, our main idea to overcome the above difficulty lies in the following:
\begin{itemize}
\item the first is to use the anti-derivative method by introducing the new
unknown function $w(t,r)$ defined by \eqref{1.20};

\item the second is to use a space weighted energy method to deduce the desired nonlinear stability result. The key point is to introduce the weight function
    $\chi(r)$ given by \eqref{1.22} to overcome the difficulties induced by the non-monotonicity of the stationary wave and the boundary condition.
\end{itemize}
Compared with that of \cite{Yang-Zhao-Zhao-2019}, we use a refined continuation argument so that we can get a nonlinear stability result only under the assumption \eqref{1.23}.

The rest of this paper is organized as follows. In Section 2, we prove Theorem \ref{Thm1.1} and the proof of Theorem \ref{Thm1.2} will be given in Section 3.

\vskip 2mm
\noindent \textbf{Notations:} We denote the usual Lebesgue space of square integrable functions over $[r_0,\infty)$ by $L^2=L^2([r_0,\infty))$ with norm $\|\cdot\|$ and for each non-negative integer $k$, we use $H^k$ to denote the corresponding $k$th-order Sobolev space $H^k([r_0,\infty))$ with norm $\|\cdot\|_{H^k}.$

For $\alpha\in \mathbb{R},$ we denote the algebraic weighted Sobolev space, that is the space of functions $f$ satisfying $ r^{\alpha/2}f\in H^k,$ by $H^{k,\alpha}$ with norm
$$
\|f\|_{k,\alpha}:=\left\| r^{\alpha/2}f\right\|_{H^k}.
$$
For $k=0,$ we denote $\|\cdot\|_{0,\alpha}$ by $|\cdot|_\alpha$ for simplicity. We also denote the exponential weighted Sobolev space, that is, the space of functions $f$ satisfying $e^{\alpha r/2}f\in H^k$ for some $\alpha\in\mathbb{R}$, by $H^{k,\alpha}_{\exp}.$ For $k=0,$ we denote $\|\cdot\|^{0,\alpha}_{\exp}$ by $|\cdot|_{\alpha,\exp}$ for simplicity. For an interval $I\subset \mathbb{R}$ and a Banach space $X,$ $C(I;X)$ denotes the space of continuous $X$-valued functions on $I,$ $C^k(I;X)$ the space of $k$-times continuously differentiable $X$-valued functions.

\section{The proof of Theorem \ref{Thm1.1}}

This section is devoted to proving Theorem \ref{Thm1.1}. To this end, we first consider the following Cauchy problem
\begin{eqnarray}\label{2.1}
\frac{d\psi}{dr}&=&\frac1{2\mu}\left(\psi^2-v_+^2+\frac{\mu^2}{r^2}\right), \quad r>r_0,\nonumber\\
\psi(r_0)&=&V_-:=v_-+\frac{\mu}{r_0}.
\end{eqnarray}

The local existence of smooth solution $\psi(r)$ to the Cauchy problem \eqref{2.1} is well-established and suppose that such a local solution $\psi(r)$ has been extended to the interval $[r_0,R)$ for some $R>r_0$, to show that such a $\psi(r)$ is indeed a solution of \eqref{1.15}, we only need to show that
\begin{itemize}
\item $\psi(r)$ can be extended step by step to the interval $[r_0,+\infty)$;
\item $\lim\limits_{r\to+\infty}\psi(r)=v_+$.
\end{itemize}

For this purpose, by exploiting the standard continuation argument, we only need to deduce certain lower and upper bounds for $\psi(r)$ on the interval $[r_0,R)$. To yield the desired lower bound for $\psi(r)$ is relatively easy. In fact, we can get from the inequality
$$
\frac{1}{2\mu}\left(\psi^2-v_+^2+\frac{\mu^2}{r^2}\right)
\geq \frac{1}{2\mu}\left(2v_+\left(\psi-v_+\right)+\frac{\mu^2}{r^2}\right)
$$
that
$$
\frac{d\psi(r)}{dr}\geq \frac{1}{2\mu}\left(2v_+\left(\psi(r)-v_+\right)+\frac{\mu^2}{r^2}\right),\quad r\in[r_0,R).
$$
From which and \eqref{2.1}$_2$, we can easily deduce that
\begin{eqnarray}\label{2.2}
\psi(r)&\geq& v_++\left(V_--v_+\right)\exp\left(-\frac{|v_+|(r-r_0)}{\mu}\right)
+\frac{\mu}{2}\int_{r_0}^r\frac1{s^2}\exp\left(-\frac{|v_+|(r-s)}{\mu}\right)ds\nonumber\\
&:=&\psi^2_S(r),\quad \forall r\in[r_0,R).
\end{eqnarray}
Moreover, one can find that $\psi^2_S(r)$ given by \eqref{1.13} with $n=2$ is well-defined for $r\geq r_0$ and satisfies
\begin{equation}\label{2.3}
\left|\psi^2_S(r)-v_+\right|\lesssim r^{-2},\quad r\geq r_0.
\end{equation}

Now we turn to find an upper bound for $\psi(r)$ on the interval $[r_0,R)$. Since the last term $-\frac{\mu(n-1)(n-3)}{2r^2}$ in the right hand side of \eqref{1.7}$_1$ has different sign for $n\geq 4$ and for $n=2$, the method used in \cite{Hashimoto-Matsumura-JDE-2019}, which has been proven to be effective for $n\geq 4$, cannot be applied to the two-dimensional case any more. More precisely, even for the case when $v_+<0, V_-<|v_+|$, the function $\psi^S(r)$ defined by \eqref{1.11} is not an upper bound for $\psi(r)$ on the interval $[r_0,R)$. Even so, one can easily deduce that for the case when $v_-<0$, the function $\psi_1^2(r)$ defined by
\begin{equation}\label{2.4}
\psi_1^2(r)=\phi^2_1(r)-\frac\mu r:=\frac{v_-}{r\left[\frac{1}{r_0}-\frac{v_-}{2\mu}\ln\frac{r}{r_0}\right]}-\frac\mu r,
\end{equation}
with $\phi^2_1(r)$ being defined by \eqref{1.8}, indeed gives a upper bound for $\psi(r)$, that is
\begin{equation}\label{2.5}
\psi(r)\leq \psi_1^2(r),\quad r_0\leq r<R.
\end{equation}

Noticing that both $\psi^2_S(r)$ defined by \eqref{1.13} and $\psi_1^2(r)$ given by \eqref{2.4} are defined on $r\geq r_0$ and are uniformly bounded on $[r_0,+\infty)$ which follows from the estimates \eqref{2.3} and \eqref{1.10},one can thus deduce that the above local solution $\psi(r)$ can indeed be extended step by step to the interval $[r_0,+\infty)$. Even so, the problem is that for the case $v_+<0$, since $\lim\limits_{r\to\infty}\psi^2_1(r)=0$, one can not deduce that $\lim\limits_{r\to\infty}\psi(r)=v_+$ even though one can show that such a $\psi(r)$ is an uniformly Lipschitz continuous function on the interval $[r_0,+\infty)$.

To overcome such a difficulty, for each $r_1\geq\max\left\{r_0, \frac\mu{|v_+|}\right\}$ (the existence of such a $r_1$ is guaranteed by the assumption that $v_+<0$), we first consider the following auxiliary Cauchy problem
\begin{eqnarray}\label{2.6}
\frac{d\eta(r;r_1)}{dr}&=&\frac{1}{2\mu}\left(\eta^2(r;r_1)-v_+^2+\frac{\mu^2}{r^2}\right),\nonumber\\
\eta(r;r_1)\big|_{r=r_1}&=&\sqrt{v_+^2-\frac{\mu^2}{r_1^2}}
\end{eqnarray}
and we want to show that the Cauchy problem \eqref{2.4} admits a unique solution $\eta(r;r_1)$ on the interval $[r_0,+\infty)$.

To do so, one can first deduce from the well-known local solvability result for the Cauchy problem of ordinary differential equations again that \eqref{2.6} admits a unique smooth solution $\eta(r;r_1)$ which are defined on the interval $(R_1, R_2)$ with $r_0\leq R_1<r_1<R_2\leq+\infty$ (For the case when $r_1=r_0$, the modification is straightforward, we only need to replace the interval $(R_1,R_2)$ by $[r_0,R_2)$). Our next lemma tells that $r=r_1$ is the global maximum point of $\eta(r;r_1)$ on the interval $(R_1, R_2)$.
\begin{lemma}\label{Lemma2.1} Suppose that $\eta(r;r_1)$ is a smooth solution of the Cauchy problem \eqref{2.6} defined on the interval $(R_1, R_2)$, then $\eta(r;r_1)$ is monotonic increasing for $R_1<r\leq r_1$ and monotonic decreasing for $r_1\leq r<R_2$, thus
\begin{equation}\label{2.7}
\sup\limits_{R_1<r<R_2}\eta(r;r_1)=\eta(r_1;r_1)=\sqrt{v_+^2-\frac{\mu^2}{r_1^2}}.
\end{equation}
\end{lemma}
\begin{proof} We only prove that $\eta(r;r_1)$ is monotonic decreasing for $r_1\leq r<R_2$. For this purpose, we first get from \eqref{2.6} that
\begin{eqnarray}\label{2.8}
\left.\frac{d\eta(r;r_1)}{dr}\right|_{r=r_1}&=&\frac{1}{2\mu}\left(\eta^2(r_1;r_1)-v_+^2 +\frac{\mu^2}{r^2_1}\right)=0,\nonumber\\
\left.\frac{d^2\eta(r;r_1)}{dr^2}\right|_{r=r_1}&=&\left.\frac{1}{\mu}\left(\eta(r;r_1) \frac{d\eta(r;r_1)}{dr}-\frac{\mu^2}{r^3}\right)\right|_{r=r_1}\\
&=&-\frac{\mu}{r^3_1}<0.\nonumber
\end{eqnarray}

\eqref{2.8} tells us that $\frac{d\eta(r;r_1)}{dr}<0$ holds for all $r_1<r\leq r_1+\varepsilon$. Here $\varepsilon>0$ is a suitably chosen sufficiently small positive constant. Now if we set
\begin{equation}\label{2.9}
R^*:=\sup\left\{r\in (r_1, R_2): \frac{d\eta(s;r_1)}{dr}\leq 0,\quad \forall s\in(r_1,r)\right\},
\end{equation}
then we can get that $R^*\in [r_1+\varepsilon, R_2)$.

We now show that $R^*=R_2$. Otherwise, if $R^*<R_2$, we can deduce that
\begin{equation}\label{2.10}
\left.\frac{d\eta(r;r_1)}{dr}\right|_{r=R^*}=0.
\end{equation}

\eqref{2.10} together with \eqref{2.6} imply
\begin{eqnarray}\label{2.11}
\left.\frac{d^2\eta(r;r_1)}{dr^2}\right|_{r=R^*}&=&\left.\frac{1}{\mu}\left(\eta(r;r_1) \frac{d\eta(r;r_1)}{dr}-\frac{\mu^2}{r^3}\right)\right|_{r=R^*}\\
&=&-\frac{\mu}{\left(R^*\right)^3}<0,\nonumber
\end{eqnarray}
from which one can further deduce that there exists a sufficiently small positive constant $r^*>0$ such that $\frac{d\eta(r;r_1)}{dr}<0$ holds for all $r\in(R^*, R^*+r^*)$, but this fact contradicts the definition of $R^*$ and consequently $R^*=R_2$. This completes the proof of Lemma \ref{Lemma2.1}.
\end{proof}

From Lemma \ref{Lemma2.1}, one can deduce the following estimate on the upper bound of the solution $\eta(r;r_1)$ of the Cauchy problem \eqref{2.6}
\begin{equation}\label{2.12}
\eta(r;r_1)\leq \eta(r_1;r_1)=\sqrt{v_+^2-\frac{\mu^2}{r_1^2}}<|v_+|,\quad R_1<r<R_2.
\end{equation}

To get an estimate on the lower bound of $\eta(r;r_1)$, we need the following comparison principle for the Cauchy problem \eqref{2.6}, which will be frequently used in this section.
\begin{lemma}[Comparison Principle]
  Let
$$y_1<y_2=y_3, \quad \delta_{ij}=\left\{
  \begin{array}{rl}
     1,\ &i=j,\\
     0,\ &i\neq j,
  \end{array}\right.
$$
and $r_0\leq r_1\leq r_2.$ Suppose that $\eta_i(r;r_1)\in C^1([r_0,r_2])$ is the solution of
\begin{eqnarray}\label{2.13}
\frac{d\eta_{i}(r;r_1)}{dr}&=&\frac1{2\mu}\left(\eta_i(r;r_1)^2-\left(1-\delta_{i3}\right)v_+^2 +\frac{\mu^2}{r^2}\right),\quad r_0\leq r\leq r_2,\nonumber\\
\eta_i(r_1;r_1)&=&y_i
\end{eqnarray}
for $i=1,2,3.$ Then we have
\begin{itemize}
\item [(i).] $\eta_1(r;r_1)\leq\eta_2(r;r_1)$ holds for $r\in [r_0,r_2]$;
\item [(ii).] $\eta_2(r;r_1)\geq\eta_3(r;r_1)$ holds for $r\in [r_0,r_1],$ while $\eta_2(r;r_1)\leq\eta_3(r;r_1)$ holds for $r\in [r_1,r_2]$.
\end{itemize}
\end{lemma}
Since the proof of Lemma 2.2 is standard, we omit the details for brevity.

Now we turn to deduce the desired lower bound estimate for $\eta(r;r_1)$. To this end, we define
\begin{equation}\label{2.14}
\overline{\eta}(r;r_1)=v_+ +\left(\eta(r_1;r_1)-v_+\right)\exp\left(-\frac{|v_+|(r-r_1)}{\mu}\right) +\frac{\mu}{2}\int_{r_1}^r\frac1{s^2}\exp\left(-\frac{|v_+|(r-s)
     }{\mu}\right)ds,\quad r\geq r_1
\end{equation}
and
\begin{equation}\label{2.15}
\widetilde{\eta}(r;r_1)=\frac 1r \left(\frac{\left(r_1\eta(r_1;r_1)+\mu\right)} {\frac1{2\mu}\left(r_1\eta(r_1;r_1)+\mu\right)\ln\frac{r_1}{r}+1}-\mu\right),\quad r_0\leq r\leq r_1,
\end{equation}
which solve the following Cauchy problems
\begin{eqnarray}\label{2.16}
\frac{d\overline{\eta}(r;r_1)}{dr}&=&\frac{v_+}{\mu}\left(\overline{\eta}(r;r_1)-v_+\right) +\frac{\mu}{2r^2},\quad r\geq r_1,\nonumber\\
\overline{\eta}(r_1;r_1)&=&\eta(r_1;r_1)=\sqrt{v_+^2-\frac{\mu^2}{r_1^2}}
\end{eqnarray}
and
\begin{eqnarray}\label{2.17}
\frac{d\widetilde{\eta}(r;r_1)}{dr}&=&\frac{\widetilde{\eta}(r;r_1)^2}{2\mu} +\frac{\mu}{2r^2}, \quad r_0\leq r\leq r_1,\nonumber\\
\widetilde{\eta}(r_1;r_1)&=&\eta(r_1;r_1)=\sqrt{v_+^2-\frac{\mu^2}{r_1^2}},
\end{eqnarray}
respectively.

Since $\eta(r_1;r_1)=\sqrt{v_+^2-\frac{\mu^2}{r_1^2}}>0$, one can deduce that  $\widetilde{\eta}(r;r_1)$ is well-defined for $r_0\leq r\leq r_1$, while $\overline{\eta}(r;r_1)$ is well-defined for $r\geq r_1$. Moreover, one can further deduce that
\begin{eqnarray}\label{2.18}
\left|\overline{\eta}(r;r_1)-v_+\right|&\lesssim& r^{-2},\quad r\geq r_1,\nonumber\\
\overline{\eta}(r;r_1)&\leq& \sqrt{v_+^2-\frac{\mu^2}{r_1^2}}<|v_+|,\quad r\geq r_1,\\
\widetilde{\eta}(r;r_1)&\geq & \widetilde{\eta}(r_0;r_1)\geq -\frac\mu{r_0},\quad r_0\leq r\leq r_1\nonumber
\end{eqnarray}
and by employing Lemma 2.2, the solution $\eta(r;r_1)$ of the Cauchy problem \eqref{2.6} satisfies the following lower bound estimate
\begin{equation}\label{2.19}
\eta(r;r_1)\geq \left\{
\begin{array}{rl}
\widetilde{\eta}(r;r_1),\quad & R_1<r\leq r_1,\\[1mm]
\overline{\eta}(r;r_1),\quad & r_1\leq r<R_2.
\end{array}
\right.
\end{equation}

Combining the estimates \eqref{2.12}, \eqref{2.18} and \eqref{2.19}, one can deduce that the solution $\eta(r,;r_1)$ of the Cauchy problem \eqref{2.6} constructed above can be indeed extended to the interval $[r_0,+\infty)$. Moreover we can further obtain from Lemma \ref{Lemma2.1} that $\eta(r;r_1)$ is uniformly bounded in $[r_0,+\infty)$ and is monotonic decreasing for $r\geq r_1$ and consequently there exists a constant $\ell\in\mathbb{R}$ such that
\begin{equation}\label{2.20}
\lim\limits_{r\to+\infty}\eta(r;r_1)=\ell,\quad \lim\limits_{r\to+\infty}\frac{d\eta(r;r_1)}{dr}=0.
\end{equation}

Having obtained \eqref{2.20}, one can get by taking the limit as $r\to+\infty$ in \eqref{2.6}$_1$ that
\begin{equation}\label{2.21}
\ell^2=|v_+|^2.
\end{equation}
From which and the estimate \eqref{2.12}, we can finally deduce that $\ell=v_+$. Thus we have proved that the solution $\eta(r;r_1)$ constructed above satisfies $\lim\limits_{r\to+\infty}\eta(r;r_1)=v_+$.

Now we prove that $\eta(r;r_1)$ satisfies
\begin{equation}\label{2.22}
\left|\eta(r;r_1)-v_+\right|\lesssim r^{-2}
\end{equation}
for $r\geq r_0$.

To prove \eqref{2.22}, we first get from the fact $\lim\limits_{r\to+\infty}\eta(r;r_1)=v_+$ that there exists a sufficiently large $r^*>r_0$ such that
\begin{equation}\label{2.23}
\frac{3v_+}{2}\leq \eta(r;r_1)\leq \frac{v_+}{2}
\end{equation}
holds for $r\geq r^*$.

On the other hand, noticing that \eqref{2.6}$_1$ can be rewritten as
\begin{equation*}
\frac{d(\eta(r;r_1)-v_+)}{dr}=\frac{\eta(r;r_1)+v_+}{2\mu}(\eta(r;r_1)-v_+)+\frac{\mu}{2r^2},\quad r\geq r_0,
\end{equation*}
we can get that
\begin{eqnarray}\label{2.24}
\eta(r;r_1)-v_+&=&\left(\eta(r^*;r_1)-v_+\right) \exp\left(\frac{1}{2\mu}\int^r_{r^*}\left(\eta(s;r_1)+v_+\right)ds\right)\\
&&+\frac{\mu}{2}\int^r_{r^*}\frac{1}{s^2}\exp\left(\frac{1}{2\mu}\int^r_{s} \left(\eta(z;r_1)+v_+\right)dz\right)ds,\quad r\geq r^*.\nonumber
\end{eqnarray}

Having obtained \eqref{2.23} and \eqref{2.24}, the estimate \eqref{2.22} follows immediately from the assumption that $v_+<0$. Thus we have proved that
\begin{lemma}\label{Lemma2.3} Let $v_+<0$, then for each $r_1\geq\max\left\{r_0, \frac\mu{|v_+|}\right\}$, the Cauchy problem \eqref{2.6} admits a unique smooth solution $\eta(r;r_1)$ which is defined on the interval $[r_0,+\infty)$ and satisfies \eqref{2.22}.
\end{lemma}

For each $r_1\geq\max\left\{r_0, \frac\mu{|v_+|}\right\}$, if we set
\begin{equation}\label{2.25}
a(r_1):=\eta(r_0;r_1),
\end{equation}
it is easy to prove that $a(r_1)$ is a smooth, increasing function of $r_1$. Moreover, the estimate \eqref{2.7} obtained in Lemma \ref{Lemma2.1} implies that
\begin{equation}\label{2.26}
a(r_1)\leq \eta(r_1;r_1)=\sqrt{v_+^2-\frac{\mu^2}{r_1^2}}.
\end{equation}
Thus there exists a constant $a_*\in\mathbb{R}$ such that $a_*=\lim\limits_{r_1\to+\infty}a(r_1)$. From the estimate \eqref{2.19}, one can easily deduce that $a_*$ satisfies \eqref{1.16}.

For the convergence of $\eta(r;r_1)$ as $r_1\to+\infty$, we first get from the estimates \eqref{2.18} and \eqref{2.19} that $\eta(r;r_1)$ satisfies the following estimate
\begin{equation}\label{2.27}
-\frac{\mu}{r_0}\leq \eta(r;r_1)\leq |v_+|,\quad r\geq r_0
\end{equation}
with its lower and upper bounds independent of the parameter $r_1$。

The estimate \eqref{2.27} together with the fact that $\eta(r;r_1)$ solves the following Cauchy problem
\begin{eqnarray}\label{2.28}
\frac{d\eta(r;r_1)}{dr}&=&\frac{1}{2\mu}\left(\eta^2(r;r_1)-v_+^2 +\frac{\mu^2}{r^2}\right),\quad r\geq r_0,\nonumber\\
\eta(r;r_1)\big|_{r=r_0}&=&a(r_1)
\end{eqnarray}
tell us that, for each $k\in\mathbb{N}$, there exists a positive constant $C_k$ independent of the parameter $r_1$ such that
\begin{equation}\label{2.29}
\sup\limits_{r\geq r_0}\left\{\left|\frac{d^k\eta(r;r_1)}{dr^k}\right|\right\}\leq C_k.
\end{equation}
Thus there exists a function $\eta^*(r;a_*)\in C^\infty ([r_0,+\infty))$ such that, for each $k\in\mathbb{N}$, $\frac{d^k\eta(r;r_1)}{dr^k}$ converges uniformly locally in $[r_0,+\infty)$ to $\frac{d^k\eta^*(r;a_*)}{dr^k}$ as $r_1\to+\infty$.

Moreover, since
$$
\lim\limits_{r_1\to+\infty}\eta(r_1;r_1) =\lim\limits_{r_1\to+\infty}\sqrt{v_+^2-\frac{\mu^2}{r_1^2}}
=|v_+|,
$$
one can show that $\eta^*(r;a_*)$ satisfies
\begin{eqnarray}\label{2.30}
\frac{d\eta^*(r;a_*)}{dr}&=&\frac{1}{2\mu}\left(\left(\eta^*(r;a_*)\right)^2-v_+^2 +\frac{\mu^2}{r^2}\right),\quad r\geq r_0,\nonumber\\
\eta^*(r;a_*)\big|_{r=r_0}&=&a_*,\quad \lim\limits_{r\to+\infty}\eta^*(r;a_*)=|v_+|.
\end{eqnarray}

With the above preparations in hand, we now turn to prove our main result Theorem \ref{Thm1.1}. Our main idea is to deduce a suitable upper bound on the solution $\psi(r)$ of the Cauchy problem \eqref{2.1} on the interval $[r_0,R)$ in which $\psi(r)$ is defined.

For the case when $V_-\geq a_*$, we first consider the Cauchy problem \eqref{2.1} and suppose that its unique solution $\psi(r)$ is defined on the interval $[r_0,R)$ for some $R>r_0$, since $V_-\geq a_*$, we can get from Lemma 2.2 that the following estimate
\begin{equation}\label{2.31}
\psi(r)\geq \eta^*(r;a_*),\quad r_0\leq r<R
\end{equation}
holds. In such a case, even if such a $\psi(r)$ can be extended to the interval $[r_0,+\infty)$ (consequently the estimate \eqref{2.13} also holds for $r\geq r_0$), we can only deduce that $\lim\limits_{r\to+\infty}\psi(r)\geq |v_+|$ since $\lim\limits_{r\to+\infty}\eta^*(r;a_*)=|v_+|$. Thus such a $\psi(r)$ can not satisfy \eqref{1.15}.

On the other hand, for the case when $V_-<a_*$, we can take $\tilde{r}_1\geq\max\left\{r_0, \frac\mu{|v_+|}\right\}$ sufficiently large such that
\begin{equation}\label{2.32}
V_-<a\left(\tilde{r}_1\right):=\eta\left(r_0;\tilde{r}_1\right)<a_*.
\end{equation}
For the above chosen $a\left(\tilde{r}_1\right)$, we have from Lemma \ref{Lemma2.3} that the following Cauchy problem
\begin{eqnarray}\label{2.33}
\frac{d\eta\left(r;\tilde{r}_1\right)}{dr} &=&\frac{1}{2\mu}\left(\left(\eta\left(r;\tilde{r}_1\right)\right)^2-v_+^2 +\frac{\mu^2}{r^2}\right),\quad r\geq r_0,\nonumber\\
\eta\left(r;\tilde{r}_1\right)\big|_{r=r_0}&=&a\left(\tilde{r}_1\right)
\end{eqnarray}
admits a unique solution $\eta\left(r;\tilde{r}_1\right)$ for $r\geq r_0$ and satisfies
\begin{equation}\label{2.34}
\lim\limits_{r\to+\infty}\eta\left(r;\tilde{r}_1\right)=v_+
\end{equation}
and
\begin{equation}\label{2.35}
\left|\eta\left(r;\tilde{r}_1\right)-v_+\right|\lesssim r^{-2},\quad r\geq r_0.
\end{equation}
Furthermore, Lemma 2.2 together with the assumption $V_-<a_*$ imply
\begin{equation}\label{2.36}
\psi(r)\leq \eta\left(r;\tilde{r}_1\right),\quad r\in[r_0,R).
\end{equation}

Now the estimates \eqref{2.2} and \eqref{2.36} provide the desired lower and upper bounds for the solution $\psi(r)$ of the Cauchy problem \eqref{2.1}, since both $\psi_S^2(r)$ given by \eqref{2.2} and $\eta\left(r;\tilde{r}_1\right)$ defined by the Cauchy problem \eqref{2.33} satisfy the estimates \eqref{2.3} and \eqref{2.35}, one can deduce immediately that such a $\psi(r)$ can be extended to the interval $[r_0,+\infty)$ and satisfies the estimate \eqref{1.14}. This completes the proof of Theorem \ref{Thm1.1}.

\section{The proof of Theorem \ref{Thm1.2}}

To prove Theorem \ref{Thm1.2}, for some positive constants $T>0$ and $M>0$, we first define the set of functions for which we seek the solution of the initial-boundary value problem (\ref{1.21}) by
\begin{equation*}
  X_M(0,T)=\left\{w(t,r)\in C\left([0,T];H^2([r_0,+\infty))\right), w_r(t,r)\in L^2\left([0,T];H^2([r_0,+\infty))\right),\ \sup\limits_{t\in [0,T]}\|w(t)\|_{L^\infty}\leq  M\right\}.
\end{equation*}
For the the local solvability of the initial-boundary value problem \eqref{1.21} in $X_M(0,T)$, we have

\begin{proposition}[Local solvability result]\label{Prop3.1}
  Assume that $w_0\in H^2$ with $\|w_0\|_{L^\infty}\leq M$ , then there exists a sufficiently small positive constant $t_1=t_1(M)$, which depends only on $M$, such that the initial-boundary value problem (\ref{1.21}) has a unique solution $w(t,r)\in X_{2M}(0,t_1).$
\end{proposition}

Proposition \ref{Prop3.1} can be proved by a standard iterative method, so we omit the details for brevity.

Suppose that the local solution $w(t,r)$ constructed in Proposition \ref{Prop3.1} has been extended to the time step $t=T\geq t_1$ and satisfies $w(t,r)\in X_N(0,T)$ for some positive constant $N\geq 2M$, our second result is concerned with the desired $H^2-${\it a priori} estimates on such a local solution $w(t,r)$ based on the following {\it a priori} assumption
\begin{equation}\label{3.1}
|w(t,r)|\leq \mu,\quad (t,r)\in[0,T)\times[r_0,+\infty).
\end{equation}
\begin{proposition}[A Priori Estimates]\label{Prop3.2}
  Suppose that $w(t,r)\in X_{N}(0,T)$ is a solution of (\ref{1.21}) which satisfies the {\it a priori} assumption \eqref{3.1}, then for each $0\leq t<T$, we can get that
\begin{equation}\label{3.2}
\|w(t)\|_\chi^2+\mu\int_0^t\left(\left\|\frac{w(\tau)}{r}\right\|^2+\|w_r(\tau)\|_\chi^2+\frac 1{r_0}w^2(\tau,r_0)\right)d\tau\leq \|w_0\|_\chi^2,
\end{equation}
\begin{equation}\label{3.3}
  \|w_r(t)\|^2+\mu\int_0^t\|w_{rr}(\tau)\|^2d\tau\leq\|w_{0r}\|^2+\frac{v_+^2}{\mu^2C_L} \|w_0\|^2_\chi,
\end{equation}
and
\begin{equation}\label{3.4}
\left\|w_{rr}(t)\right\|^2+\mu\int_0^\infty\left\|w_{rrr}(\tau)\right\|^2d\tau\lesssim 1+\left\|w_{0}\right\|^2_{H^2}.
\end{equation}
\end{proposition}
\begin{proof} We first prove \eqref{3.2}. To this end, as in \cite{Yang-Zhao-Zhao-2019}, we can get by multiplying $(\ref{1.21})_1$ by $\chi w$ and the definition of the weight $\chi(r)$ that
\begin{equation}\label{3.5}
\left(\frac 12\chi w^2\right)_t+\frac\mu2\frac{w^2}{r^2}+\mu\chi w_r^2 -\mu\left(\chi ww_r+\left(\frac1{r_0}-\frac1{2r}\right)w^2\right)_r=-\frac 12\chi ww_r^2.
\end{equation}

Integrating the above identity with respect to $r$ over $(r_0,\infty)$ and noticing that the {\it a priori} assumption \eqref{3.1} tells us that
\begin{equation*}
    -\frac 12\int_{r_0}^\infty\chi ww^2_rdr\leq \frac12\|w\|_{L^\infty}\int_{r_0}^\infty\chi w^2_rdr
    <\leq\frac\mu 2\int_{r_0}^\infty\chi w^2_rdr,
\end{equation*}
we obtain
  \begin{equation*}
    \frac 12\frac {d}{dt}\|w(t)\|_\chi^2 +\frac{\mu}2\left\|\frac{w(t)}{r}\right\|^2+\frac\mu2\|w_r(t)\|_\chi^2+\frac \mu{2r_0}w^2(t,r_0)\leq 0.
  \end{equation*}
Thus \eqref{3.2} follows immediately by integrating the above inequality with respect to $t$.

Now we turn to prove \eqref{3.3}. For this purpose, differentiating $(\ref{1.21})_1$ with respect to $r$ once and multiplying the resulting identity by $w_r,$ we can obtain from \eqref{1.15}$_1$ that
\begin{eqnarray*}
\left(\frac12 w_r^2\right)_t-\mu w_rw_{rrr}&=&-\psi w_rw_{rr}-\left(\frac13w_r^3\right)_r-\psi_rw_r^2\nonumber\\
  &\leq&\frac\mu2w_{rr}^2+\left(\frac{\psi^2}{2\mu}-\psi_r\right)w_r^2 -\left(\frac13w_r^3\right)_r\\
  &\leq&\frac\mu2w_{rr}^2+\frac{v_+^2}{2\mu}w_r^2-\left(\frac13w_r^3\right)_r.\nonumber
\end{eqnarray*}

Integrating the above inequality with respect to $r$ and $t$  over $[r_0,\infty)$ and $[0,t],$ we obtain
\begin{equation}\label{3.6}
   \|w_r(t)\|^2+\mu\int_0^t\|w_{rr}(\tau)\|^2d\tau\leq\|w_{0r}\|^2+\frac{v_+^2}{\mu}\int_0^t\|w_r(\tau)\|^2d\tau.
\end{equation}
On the other hand, \eqref{3.2} implies that
\begin{equation}\label{3.7}
  \int_0^t\|w_r(\tau)\|^2d\tau\leq\frac1{\mu C_L}\|w_0\|^2_\chi,
\end{equation}
from which and \eqref{3.6}, we can deduce \eqref{3.3} immediately.

Finally, we deduce the second-order energy type estimates. To do so, differentiating $(\ref{1.21})_1$ with respect to $r$ twice yields
  \begin{equation}\label{3.8}
    w_{rrt}-\mu w_{rrrr}=-w_{rr}^2-w_rw_{rrr}-\psi_{rr}w_r-2\psi_rw_{rr}-\psi w_{rrr}.
  \end{equation}
Multiplying (\ref{3.8}) by $w_{rr}$ and integrating the resulting equation with respect to $r$ over $[r_0,\infty),$ we get that
  \begin{eqnarray}\label{3.9}
    &&\frac 12\frac d{dt}\left\|w_{rr}(t)\right\|^2+\mu\left\|w_{rrr}(t)\right\|^2\nonumber\\
    &\leq& \underbrace{-\mu w_{rr}(t,r_0)w_{rrr}(t,r_0)}_{I_1} \underbrace{-\int_{r_0}^\infty\psi w_{rr}w_{rrr} (t,r)dr}_{I_2}\\
    &&\underbrace{-\int_{r_0}^\infty\left(\psi_{rr}w_rw_{rr} +2\psi_rw^2_{rr}\right)(t,r)dr}_{I_3}\underbrace{-\int_{r_0}^\infty \left( w_{rr}^3+w_rw_{rr}w_{rrr}\right)(t,r)dr}_{I_4}.\nonumber
  \end{eqnarray}

Now we turn to estimate the terms $I_j(j=1,2,3,4)$ in the right hand side of \eqref{3.9}. For $I_1$, if we take $r=r_0$ in (\ref{3.8}), then we can get from the facts $w_{rr}(t,r_0)=0, w_r(t,r_0)=0, \psi(r_0)=V_-$ that
  \begin{equation*}
    w_{rrr}(t,r_0)=\frac{V_-}{\mu}w_{rr}(t,r_0).
  \end{equation*}
Consequently, from the above identity and the Sobolev inequality, $I_1$ can be estimated as
  \begin{eqnarray}\label{3.10}
      I_1&\leq&\left|w_{rr}(t,r_0)w_{rrr}(t,r_0)\right|\leq C\left|w_{rr}(t,r_0)\right|^2\nonumber\\
       &\leq& C\left\|w_{rr}(t)\right\|\left\|w_{rrr}(t)\right\|
      \leq \frac{\mu}4\left\|w_{rrr}(t)\right\|^2+C\left\|w_{rr}(t)\right\|^2.
  \end{eqnarray}
For $I_2$ and $I_3,$ the Young inequality tells us
  \begin{equation}\label{3.11}
     I_2\leq \frac\mu4\left\|w_{rrr}(t)\right\|^2+C\left\|w_{rr}(t)\right\|^2
  \end{equation}
  and
  \begin{equation}\label{3.12}
    I_3\leq C\left(\left\|w_r(t)\right\|^2+\left\|w_{rr}(t)\right\|^2\right).
  \end{equation}
For $I_4$, the Sobolev inequality, the H\"{o}lder inequality and the Young inequality tell us that
\begin{eqnarray}\label{3.13}
  I_4&=&\int_0^\infty w_rw_{rr}w_{rrr}dr\leq\|w_r\|_{L^\infty}\|w_{rr}\|\|w_{rrr}\|\nonumber\\
  &\leq&\frac\mu4\|w_{rrr}\|^2+\frac1\mu\|w_r\|_{L^\infty}^2\|w_{rr}\|^2\leq\frac\mu4\|w_{rrr}\|^2+\frac2\mu\|w_r\|\|w_{rr}\|^3\\
  &\leq&\frac\mu4\|w_{rrr}\|^2+\frac1\mu\left(\|w_r\|^2+\|w_{rr}\|^2\right)\|w_{rr}\|^2.\nonumber
\end{eqnarray}

Substituting the estimates (\ref{3.10})-(\ref{3.13}) into (\ref{3.9}), we have
\begin{equation}\label{3.14}
  \frac{d}{dt}\|w_{rr}(t)\|^2+\mu\|w_{rrr}(t)\|^2\lesssim\|w_r(t)\|^2+\|w_{rr}(t)\|^2+(\|w_r(t)\|^2+\|w_{rr}\|^2)\|w_{rr}\|^2.
\end{equation}

\eqref{3.14}, \eqref{3.2}, \eqref{3.3} together with the Gr\"{o}nwall inequality imply
\eqref{3.4} holds. This completes the proof of Proposition \ref{Prop3.2}.
\end{proof}

Having obtained Proposition \ref{Prop3.1} and Proposition \ref{Prop3.2}, we now turn to prove Theorem \ref{Thm1.2} by the continuation argument. In fact, the assumption \eqref{1.23} imposed on the initial data $w_0(r)$ together with Sobolev's inequality imply that
\begin{equation}\label{3.15}
\|w_0\|_{L^\infty}\leq \frac\mu 2,
\end{equation}
which guarantees the local existence of solution $w(t,r)\in X_\mu(0,T)$ for some $T>0$. On the other hand, from the {\it a priori} estimates \eqref{3.2}, \eqref{3.3} and \eqref{3.4} obtained in Proposition \ref{Prop3.2}, one can get that
\begin{eqnarray}\label{3.16}
  \|w(t)\|_{L^\infty}&\leq& \sqrt2\|w(t)\|^{\frac12}\|w_r(t)\|^{\frac12}\nonumber\\
  &\leq&\sqrt2C_U^{\frac14}C_L^{-\frac14}\|w_0\|^{\frac12}\bigg(\|w_{0rr}\|^2+\frac{C_Uv_+^2}{C_L\mu^2}\|w_0\|^2\bigg)^{\frac14}\\
  &\leq&\frac\mu2\nonumber
\end{eqnarray}
holds for $0\leq t\leq T$ provided that the assumption \eqref{1.23} holds and consequently, by employing Proposition \ref{Prop3.1} again, $w(t,r)$ can be extended to the time interval $t=T+t_1$ and $w(t,r)\in X_\mu(0,T+t_1)$. Repeating the above procedure and by the continuation argument, one can thus extend $w(t,r)$ step by step to a global one. This completes the proof of Theorem \ref{Thm1.2}.

Since Theorem \ref{Thm1.3} can be proved by repeating the argument used in \cite{Yang-Zhao-Zhao-2019} to prove Theorem 2.2 and Theorem 2.3 there, we thus omit for brevity.

\section{Acknowledgements} The work was supported by the Fundamental Research Funds for the Central Universities and two grants from the National Natural Science Foundation of China under contracts 11731008 and 11671309, respectively.

\end{document}